\documentclass[12pt,oneside,reqno]{amsart}
\usepackage{mathrsfs}
\usepackage{graphics}
\usepackage{enumerate, amssymb}
\newcommand{\dif}{\mathrm{d}}

\newcommand{\be}{\begin{eqnarray}}
\newcommand{\ee}{\end{eqnarray}}
\newcommand{\ce}{\begin{eqnarray*}}
\newcommand{\de}{\end{eqnarray*}}
\newtheorem{theorem}{Theorem}[section]
\newtheorem{lemma}[theorem]{Lemma}
\newtheorem{remark}[theorem]{Remark}
\newtheorem{definition}[theorem]{Definition}
\newtheorem{proposition}[theorem]{Proposition}
\newtheorem{example}[theorem]{Example}
\newtheorem{corollary}[theorem]{Corollary}

\def\[{{\Big[}}
\def\]{{\Big]}}
\def\<{{\langle}}
\def\>{{\rangle}}
\def\({{\Big(}}
\def\){{\Big)}}

\def\no{\nonumber}
\def\bt{\begin{theorem}}
\def\et{\end{theorem}}
\def\bl{\begin{lemma}}
\def\el{\end{lemma}}
\def\br{\begin{remark}}
\def\er{\end{remark}}
\def\bx{\begin{Example}}
\def\ex{\end{Example}}
\def\bd{\begin{definition}}
\def\ed{\end{definition}}
\def\bp{\begin{proposition}}
\def\ep{\end{proposition}}
\def\bc{\begin{corollary}}
\def\ec{\end{corollary}}
\def\supp{\text{\rm supp}}

\def\cF{{\mathcal F}}

\def\cH{{\mathcal H}}

\def\mE{{\mathbb E}}

\def\mN{{\mathbb N}}

\def\mP{{\mathbb P}}

\def\mR{{\mathbb R}}

\def\mU{{\mathbb U}}

\def\geq{\geqslant}
\def\leq{\leqslant}

\begin{document}

\allowdisplaybreaks

\title{Characterising the path-independent property of the Girsanov density for degenerated
stochastic differential equations}

\author{Bo Wu$^1$ and Jiang-Lun Wu$^2$}



\thanks{*This work was partly supported by NSF of China (No. 11371099).}


\date{}

\dedicatory{1. School of Mathematical Sciences, Fudan
University\\
 Shanghai 200433, China\\
wubo@fudan.edu.cn\\
2. Department of Mathematics, College of Science, Swansea University\\
Singleton Park, Swansea SA2 8PP, UK\\
j.l.wu@swansea.ac.uk}

\begin{abstract}
In this paper, we derive a characterisation theorem for the path-independent property of the density of the Girsanov transformation for {\it degenerated} stochastic differential equations (SDEs), extending the characterisation theorem of \cite{twwy} for the non-degenerated SDEs. We further extends our consideration to non-Lipschitz SDEs with jumps and with degenerated diffusion coefficients, which generalises the corresponding characterisation theorem established in \cite{hqwu}.
\end{abstract}

\maketitle \rm

\section{Introduction}
Let $(\Omega,\cF,\mP, \{\cF_t\}_{t\geq0})$ be a filtered probability space. Let $d,m\in\mN$ be fixed. We are concerned with the following SDE
\begin{equation}\label{eq1.1}
\dif X_t=b(t,X_t)\dif t +\sigma(t,X_t)\dif W_t, \quad t\ge0,
\end{equation}
where
$$b:[0,\infty)\times \mR^d\rightarrow\mR^d,\quad \sigma:[0,\infty)\times\mR^d\rightarrow \mR^{d\otimes m}$$
$(W_t)_{t\geq0}$ is an $m$-dimensional $\{\cF_t\}_{t\geq0}$-Brownian motion. Under standard usual conditions,
e.g. the two coefficients $b$ and $\sigma$ satisfy linear growth and local Lipschitz conditions (for the second variable), there is a unique solution to the above SDE (\ref{eq1.1}) for a given initial data $X_0$, see, e.g., \cite{iw}.

The celebrated Girsanov theorem provides a very powerful tool
to solve SDEs under the name of the {\it Girsanov transformation}
or the {\it transformation of the drift}. We use $|\cdot|$ and $\<\cdot,\cdot\>$ to denote
the Euclidean norm and scalar product of vectors in $\mR^m$ or $\mR^d$, respectively.
Let $\gamma:[0,\infty)\times \mR^d\rightarrow\mR^m$ be a measurable function such that the following exponential
integrability along the paths of the solution $(X_t)_{t\geq0}$ holds (also known as Novikov condition)
\begin{equation}\label{eq1.2}
\mE\left(\exp\left\{-\int^t_0|\gamma(s,X_s)|^2\dif s+\int^t_0\<\gamma(s,X_s), \dif W_s\>\right\}\right)<\infty,\quad t\geq0 .
\end{equation}
Then, Girsanov theorem (\cite{dm,iw,jjas}) says that for any arbitrarily fixed $T>0$
\begin{equation}\label{eq1.3}
\tilde{W}_t:=W_t-\int^t_0\gamma(s,X_s)\dif s, \quad t\in[0,T]
\end{equation}
is an $m$-dimensional $\{\cF_t\}_{t\in[0,T]}$-Brownian motion under the probability measure
\begin{equation}\label{eq1.4}
Q_T:=\exp\left\{-\int^T_0|\gamma(s,X_s)|^2\dif s+\int^T_0\<\gamma(s,X_s), \dif W_s\>\right\}\cdot \mP .
\end{equation}
Moreover, the solution $(X_t)_{t\in[0,T]}$ fulfils the following SDE
\begin{equation}\label{eq1.5}
\dif X_t=[b(t,X_t)+\sigma(t,X_t)\gamma(t,X_t)]\dif t +\sigma(t,X_t)\dif \tilde{W}_t, \quad t\in[0,T].
\end{equation}

Now let us assume that (along the paths of the solution $(X_t)_{t\geq0}$)
\begin{equation}\label{eq1.6}
b(t,X_t)-\sigma(t,X_t)\gamma(t,X_t)=0, \quad a.s. ~\forall~t\geq0.
\end{equation}
Equivalently, $b\in Im(\sigma)$, where $Im(\sigma)$ is the imagine space of $\sigma$.
Then
\begin{equation}\label{eq1.7}
\dif X_t=\sigma(t,X_t)\dif \tilde{W}_t.
\end{equation}
We are interested in the path-independent property for the exponent of the Girsanov density of $Q_T$ for any fixed $T>0$. That is, whether there exists a scalar function
$v:[0,\infty)\times \mR^d\rightarrow\mR$ such that
\begin{equation}\label{eq1.8}
Z_t:=\frac{1}{2}\int^t_0|\gamma(s,X_s)|^2\dif s+\int^t_0\<\gamma(s,X_s), \dif W_s\>
=v(t,X_t)-v(0,X_0),\quad t\ge0.
\end{equation}
This problem arises from a number of studies in economics, finance as well as from stochastic mechanics, just mention a few, see \cite{twwy,wy} (and references therein).

If $\sigma:[0,\infty)\times\mR^d\rightarrow\mR^{d\otimes d}$ (i.e., taking $m=d$) is non-degenerate, that is, the $d\times d$-matrix $\sigma(t,x)$ is invertible for any $(t,x)\in [0,\infty)\times\mR^d$,
 a characterisation of the path-independent property has obtained in \cite{twwy}.

Throughout the article, we assume that there is a unique solution to the above SDE (\ref{eq1.1}) for a given initial data $X_0$. In particular, we allow $\sigma$ be degenerate.

Let $\Lambda_t:=\{X_t(\omega)\in \mR^d:\omega\in \Omega\}\subset\mR^d$, the support of the solution. In particular, for each $t>0$, we have $\Lambda_t=\Lambda(:=\cap \Lambda_t)$ if $b(t,x)$ and $\sigma(t,x)$ satisfy the H\"{o}rmander's conditions for any $t\geq0$. Then by using It\^{o}'s formula to $v(t,X_t)$ viewing as the composition of
$v:[0,\infty)\times \mR^d\rightarrow\mR$ with the semimartingale $(X_t)_{t\geq0}$, the utilising the uniqueness of Doob-Meyer decomposition for continuous semimartingales, we can derive for any $t\geq0$ the following
\begin{equation}\label{eq1.9}\aligned
\gamma(t,X_t)=\sigma^*(t,X_t)\nabla v(t,X_t)
\endaligned\end{equation}
and
\begin{equation}\label{eq1.10}
\frac{1}{2}|\gamma(t,X_t)|^2=\frac{\partial v}{\partial t}(t,X_t)+\<\nabla v(t,X_t), b(t,X_t)\>+\frac{1}{2}Tr[(\sigma\sigma^*(t,X_t)\nabla^2v(t,X_t))]
\end{equation}
where $\sigma^*(t,x)$ stands for the transposed matrix of $\sigma(t,x)$, $\nabla$ and $\nabla^2$ stand for the gradient and Hessian operators with respect to the second variable, respectively.
Moreover, we get
\begin{equation}\label{eq1.11}\aligned
\begin{cases}
&\gamma(t,x)=\sigma^*(t,x)\nabla v(t,x)\quad(I)\\
&\frac{1}{2}|\gamma(t,x)|^2=\frac{\partial v}{\partial t}(t,x)+\<\nabla v(t,x), b(t,x)\>
+\frac{1}{2}Tr[(\sigma\sigma^*(t,x)\nabla^2v(t,x))]\quad (II)
\end{cases}
\endaligned\end{equation}
for any $(t,x)\in [0,\infty)\times \Lambda$. Putting (I) into (II) and (\ref{eq1.6}) yield the following nonlinear parabolic PDE of the (reversible) HJB type
\begin{equation}\label{eq1.12}\aligned
\begin{cases}
&\frac{\partial v}{\partial t}(t,x)=-\frac{1}{2}\{Tr[(\sigma\sigma^*(t,x)\nabla^2v(t,x))]+|\sigma^*(t,x)\nabla v(t,x)|^2\},\\
&\sigma(t,x)\sigma^*(t,x)\nabla v(t,x)=b(t,x),\quad (t,x)\in [0,\infty)\times \Lambda,
\end{cases}
\endaligned\end{equation}
where $\Lambda:=\cap \Lambda_t$.

\begin{remark}\label{r1.1} All above derivations are reciprocal, namely, that gives a characterisation of path-independence property.
\end{remark}

\begin{theorem}\label{th1.1} Assume that $\gamma:[0,\infty)\times \mR^d\rightarrow\mR^d$ is a function satisfying (\ref{eq1.6}). Then there exists a unique scalar function $v\in C^{1,2}((0,\infty)\times \mR^d\rightarrow \mR)$ such that
\begin{equation}\label{eq1.13}\aligned\frac{1}{2}\int^t_0|\gamma(s,X_s)|^2\dif s+\int^t_0\<\gamma(s,X_s), \dif W_s\>=v(t,X_t)-v(0,X_0)
\endaligned\end{equation}
if and only if (\ref{eq1.12}) holds.
\end{theorem}

\begin{proof} By the previous argument, we only show that the sufficiency. Since
$$v\in C^{1,2}((0,\infty)\times \mR^d\rightarrow \mR^d),$$
we know that $v(t,X_t)$ is a continuous semimartinagle of $X_t$. Thus we ahve
\begin{equation}\label{eq1.14}\aligned \dif v(t,X_t)&=\left\{\frac{\partial v}{\partial t}(t,X_t)+\<\nabla v(t,X_t), b(t,X_t)\>+\frac{1}{2}Tr[(\sigma\sigma^*(t,X_t)\nabla^2v(t,X_t))]\right\}\dif t\\
&\qquad\qquad +\<\sigma^*(t,X_t)\nabla v(t,X_t),\dif W_t\>.
\endaligned\end{equation}
Combining this with (\ref{eq1.12}), we get
\begin{equation}\label{eq1.15}\aligned {} & \dif v(t,X_t) \\
&=\left\{-\frac{1}{2}|\sigma^*(t,X_t)\nabla v(t,X_t)|^2+\<\nabla v(t,X_t), b(t,X_t)\>\right\}+\<\sigma^*(t,X_t)\nabla v(t,X_t),\dif W_t\>\\
&=\frac{1}{2}|\gamma(t,X_t)|^2\dif t+\<\gamma(t,X_t), \dif W_t\>.
\endaligned\end{equation}
This implies (\ref{eq1.13}). The uniqueness of $v$ is easily obtained by (eq1.13).
\end{proof}

In the following, we give some degenerated examples(see \cite{W2} for details) with satisfying H\"{o}rmander's conditions.

\begin{example}\label{exa1.3}{\bf [Gruschin operator]} Let $b(t,z)=(-xt,-x^kyt)^T, z=(x,y)\in \mR^2, t\geq0$ and $\sigma(t,z)$ be given by
\begin{equation}\label{eq1.31}\aligned \sigma(t,z)=
\left( \begin{array}{ccc}     1 & 0\\  0 & x^k\\ \end{array}\right),\quad k\in \mathbb{N}, z=(x,y)\in \mR^2, t\geq0.
\endaligned\end{equation}
Then $b\in Im(\sigma)$ and the H\"{o}rmander's condition holds for $\cH=\{\frac{\partial}{\partial x}, x^k\frac{\partial}{\partial y}\}$ with commutators up to order $k$. Define the subelliptic diffusion operator
$$L=X^2+Y^2+b(t,\cdot).$$
Let $\gamma(t,z)=(-xt,-yt)^*$ and $X_s$ be the associated $L$-diffusion process, then $b(t,z)=\sigma(t,z)\gamma(t,z)$. Assume that $v\in C^{1,2}((0,\infty)\times \mR^2\rightarrow \mR)$ fulfills the following
\begin{equation}\label{eq1.32}\aligned\frac{1}{2}\int^t_0|\gamma(s,X_s)|^2\dif s+\int^t_0\<\gamma(s,X_s), \dif W_s\>=v(t,X_t)-v(0,X_0).
\endaligned\end{equation}
Then, by Theorem \ref{th1.1}, we know that $v$ satisfies the equation (\ref{eq1.12}).
\end{example}

\begin{example}\label{exa1.4}{\bf [Kohn operator]} Consider the three-dimensional
Heisenberg group realized as $\mR^3$ equipped with the group multiplication
$$(x,y,z)(x',y',z') := (x+x', y+y', z+z'+(xy'-x'y)/2),$$
which is a Lie group with left-invariant orthonormal frame $\{X,Y,Z\}$, where
$$X =\frac{\partial}{\partial x}-\frac{y}{2}\frac{\partial}{\partial z},~ Y =\frac{\partial}{\partial y}+\frac{x}{2}\frac{\partial}{\partial z},~Z=[X,Y]=\frac{\partial}{\partial z}$$
Then the Kohn-Laplacian is $\Delta_H:=X^2 + Y^2.$
Let
$$b(t,u)=\left(xt,yt, \frac{z(x-y)}{2}t\right), \quad u=(x,y,z)\in \mR^3, t\geq0$$
and $\sigma(t,z)$ be given by
\begin{equation}\label{eq1.31}\aligned \sigma(t,u)=
\left( \begin{array}{ccc}
 1 & 0 & 0\\
 0 & 1 & 0\\
 -\frac{y}{2} & \frac{x}{2}&0\\ \end{array}\right),\quad u=(x,y,z)\in \mR^3, t\geq0.
\endaligned\end{equation}
Define the subelliptic diffusion operator
$$L=X^2+Y^2+b(t,\cdot).$$
Let $\gamma(t,z)=(xt,yt,zt)^*$ and $X_s$ be the associated $L$-diffusion process, then $b\in Im(\sigma)$ and $b(t,z)=\sigma(t,z)\gamma(t,z)$.
Then, the H\"{o}rmander's condition holds for $\cH=\{X, Y\}$ . Assume that $v\in C^{1,2}((0,\infty)\times \mR^2\rightarrow \mR)$ fulfills the following
\begin{equation}\label{eq1.32}\aligned\frac{1}{2}\int^t_0|\gamma(s,X_s)|^2\dif s+\int^t_0\<\gamma(s,X_s), \dif W_s\>=v(t,X_t)-v(0,X_0)
\endaligned\end{equation}
Then, by theorem \ref{th1.1}, we know that $v$ satisfies the equation (\ref{eq1.12}).
\end{example}

Next, we present a degenerated example without H\"{o}rmander's conditions.

\begin{example}\label{exa1.5}Let $f(x)=e^x, x\in \mR$. Denote by $X_t$ the degenerated diffusion process on $\mR^2$ given by
$$X_t=(B_t^1,e^{B_t^2})^T,$$
where $B_t^i,i=1,2$ are independent of one dimensional Brownian motion. For every $z=(x,y)\in \mR^2, t\geq0$, define
\begin{equation*}\aligned &b(t,z)=\left(0,\frac{1}{2}y\vee 0\right)^T\\
&\sigma(t,z)=
\left( \begin{array}{ccc}     1 & 0\\  0 & y\vee 0\\ \end{array}\right).
\endaligned\end{equation*}
Then $b\in Im(\sigma)$ and $\sigma(t,z),b(t,z)$ does not satisfy H\"{o}rmander's condition. And $X_s$ be the associated $L$-diffusion process for the generator $L:=\sigma\sigma^*.$
Let $\gamma(t,z)=(0,\frac{1}{2}g(\tilde{f}^{-1}(y)))^T$,
then $b(t,z)=\sigma(t,z)\gamma(t,z)$. Assume that $v\in C^{1,2}((0,\infty)\times \mR^2\rightarrow \mR)$ fulfills the following
\begin{equation}\label{eq1.32}\aligned\frac{1}{2}\int^t_0|\gamma(s,X_s)|^2\dif s+\int^t_0\<\gamma(s,X_s), \dif W_s\>=v(t,X_t)-v(0,X_0).
\endaligned\end{equation}
Then, by Theorem \ref{th1.1}, we know that $v$ satisfies the equation (\ref{eq1.12}).
\end{example}

\begin{theorem}\label{th1.3} Assume that $\gamma:[0,\infty)\times \mR^d\rightarrow\mR^d$ is a function satisfying (\ref{eq1.6}). Then there exist a function $f\in C^2(\mR\rightarrow\mR)$ and a scalar function $v\in C^{1,2}((0,\infty)\times \mR^d\rightarrow \mR)$ such that
\begin{equation}\label{eq1.16}\aligned\frac{1}{2}\int^t_0|\gamma(s,X_s)|^2\dif s+\int^t_0\<\gamma(s,X_s), \dif W_s\>=f(v(t,X_t))-f(v(0,X_0))
\endaligned\end{equation}
if and only if
\begin{equation}\label{eq1.17}\aligned
\begin{cases}
&f'(v)(t,x) \frac{\partial v}{\partial t}(t,x)=-\frac{1}{2}\big\{f'(v)(t,x) Tr[(\sigma\sigma^*(t,x)\nabla^2v(t,x))]\\
&\qquad\qquad\quad +f''(v)(t,x)|\sigma^*(t,x) \nabla v(t,x)|^2
+|f'(v)(t,x) |^2|\sigma^*(t,x)\nabla v(t,x)|^2\big\},\\
&b(t,x)=f'(v)(t,x)\sigma(t,x)\sigma^*(t,x)\nabla v(t,x),\quad (t,x)\in [0,\infty)\times \Lambda.
\end{cases}
\endaligned\end{equation}
\end{theorem}

\begin{proof} According to Theorem \ref{th1.1}, we know that (\ref{eq1.16}) is equivalent to
\begin{equation*}\aligned
\begin{cases}
&\frac{\partial f(v)}{\partial t}(t,x)=-\frac{1}{2}\{Tr[(\sigma\sigma^*(t,x)\nabla^2f(v)(t,x))]+|\sigma^*(t,x)\nabla f(v)(t,x)|^2\},\\
&\sigma(t,x)\sigma^*(t,x)\nabla f(v)(t,x)=b(t,x),\quad (t,x)\in [0,\infty)\times \Lambda.
\end{cases}
\endaligned\end{equation*}
Since
\begin{equation*}\aligned
&Tr[(\sigma\sigma^*(t,x)\nabla^2f(v)(t,x))]=Tr[(\sigma\sigma^*(t,x)\nabla(f'(v)(t,x)\nabla v(t,x))]\\
&=f'(v)(t,x) Tr[(\sigma\sigma^*(t,x)\nabla^2 v(t,x))]+f''(v)(t,x)|\sigma^*(t,x) \nabla v(t,x)|^2
\endaligned\end{equation*}
and
$$\sigma^*(t,x)\nabla f(v)(t,x)= f'(v)\sigma^*(t,x)\nabla v(t,x).$$
Combining all the above equalities, we conclude that (\ref{eq1.16}) is equivalent to (\ref{eq1.17}).
\end{proof}

\begin{remark}\label{exa1.4} Under the conditions of Theorem \ref{th1.3}, we have the following examples of the function $f$

(a) If $f(x)=x$, then
\begin{equation}\label{eq1.18}\aligned\frac{1}{2}\int^t_0|\gamma(s,X_s)|^2\dif s+\int^t_0\<\gamma(s,X_s), \dif W_s\>=v(t,X_t)-v(0,X_0)
\endaligned\end{equation}
if and only if
\begin{equation}\label{eq1.19}\aligned\sigma(t,x)\sigma^*(t,x)\nabla v(t,x)=b(t,x),\quad (t,x)\in [0,\infty)\times \Lambda
\endaligned\end{equation}
and $v$ satisfies the following time-reversed KPZ type equation,
\begin{equation}\label{eq1.20}
\frac{\partial v}{\partial t}(t,x)=-\frac{1}{2}\{Tr[(\sigma\sigma^*(t,x)\nabla^2 v(t,x))]
+|\sigma^*(t,x)\nabla v(t,x)|^2\}, \quad (t,x)\in [0,\infty)\times \Lambda.
\end{equation}
In particular, if $\sigma$ is invertible, this covers the result obtained in \cite{twwy}.

(b) Assume that $v\in C^{1,2}((0,\infty)\times \mR^d\rightarrow \mR)$ and $f(x)=\log x$ for $x>0$, then
\begin{equation}\label{eq1.21}\aligned\frac{1}{2}\int^t_0|\gamma(s,X_s)|^2\dif s+\int^t_0\<\gamma(s,X_s), \dif W_s\>=\log \frac{v(t,X_t)}{v(0,X_0)}
\endaligned\end{equation}
if and only if
\begin{equation}\label{eq1.22}\aligned\sigma(t,x)\sigma^*(t,x)\nabla v(t,x)=v(t,x)b(t,x),\quad (t,x)\in [0,\infty)\times \Lambda,
\endaligned\end{equation}
and $v$ satisfies the following time-reversed heat kernel type equation,
\begin{equation}\label{eq1.23}\aligned
\frac{\partial v}{\partial t}(t,x)&=-\frac{1}{2}Tr[(\sigma\sigma^*(t,x)\nabla^2 v(t,x))], \quad (t,x)\in [0,\infty)\times \Lambda.
\endaligned\end{equation}
In particular, if $\sigma=Id$, then we have
\begin{equation}\label{eq1.24}\aligned\frac{1}{2}\int^t_0|b(s,X_s)|^2\dif s+\int^t_0\<b(s,X_s), \dif W_s\>=\log \frac{v(t,X_t)}{v(0,X_0)}
\endaligned\end{equation}
if and only if
\begin{equation}\label{eq1.25}\aligned\nabla v(t,x)=v(t,x)b(t,x),\quad (t,x)\in [0,\infty)\times \mR^d,
\endaligned\end{equation}
and $v$ satisfies the standard heat kernel equation,
\begin{equation}\label{eq1.26}\aligned
\frac{\partial v}{\partial t}(t,x)&=-\frac{1}{2}\Delta v(t,x), \quad (t,x)\in [0,\infty)\times \mR^d.
\endaligned\end{equation}

(c) If $f(x)=x^{2k+1}, k\in \mathbb{N}\cup\{0\}$, or $f(x)=x^{2k+1}, k\in \mathbb{Z}$ for $x\neq0$, then
\begin{equation}\label{eq1.27}\aligned\frac{1}{2}\int^t_0|\gamma(s,X_s)|^2\dif s+\int^t_0\<\gamma(s,X_s),\dif W_s\>=v^{2k+1}(t,X_t)-v^{2k+1}(0,X_0)
\endaligned\end{equation}
if and only if
\begin{equation}\label{eq1.28}\aligned (2k+1) v^{2k}(t,x)\sigma(t,x)\sigma^*(t,x)\nabla v(t,x)=b(t,x),\quad (t,x)\in [0,\infty)\times \Lambda,
\endaligned\end{equation}
and $v$ satisfies the following time-reversed HJB equation,
\begin{equation}\label{eq1.29}\aligned
\frac{\partial v}{\partial t}(t,x)&=-\frac{1}{2}\Big\{Tr[(\sigma\sigma^*(t,x)\nabla^2v(t,x))]\\
&~~~~~~~~~~~~~+\frac{(2k+1) v^{2k+1}(t,x)+2k}{v(t,x)}|\sigma^*(t,x) \nabla v(t,x)|^2\Big\}.
\endaligned\end{equation}

(d) If $f(x)=\tan(x), x\in(-\frac{\pi}{2},\frac{\pi}{2})$, then
\begin{equation}\label{eq1.27}\aligned\frac{1}{2}\int^t_0|\gamma(s,X_s)|^2\dif s+\int^t_0\<\gamma(s,X_s), \dif W_s\>=\tan(v(t,X_t))-tan(v(0,X_0))
\endaligned\end{equation}
if and only if
\begin{equation}\label{eq1.28}\aligned \sigma(t,x)\sigma^*(t,x)\nabla v(t,x)=\cos^2(v(t,x)) b(t,x),\quad (t,x)\in [0,\infty)\times \Lambda,
\endaligned\end{equation}
and $v$ satisfies the following time-reversed HJB equation,
\begin{equation}\label{eq1.29}\aligned
\frac{\partial v}{\partial t}(t,x)&=-\frac{1}{2}\Big\{Tr[(\sigma\sigma^*(t,x)\nabla^2v(t,x))]\\
&~~~~~~~~~~~~~+\frac{[\cos(v(t,x))+\sin(v(t,x))]^2}{\cos^2(v(t,x))}|\sigma^*(t,x) \nabla v(t,x)|^2\Big\}.
\endaligned\end{equation}
\end{remark}

\begin{proof} It is obvious for (a). We only prove (b)((c) may be similarly handed). By Theorem \ref{th1.3}, we know that (\ref{eq1.21}) is equivalent to
\begin{equation}\label{eq1.30}\aligned
\begin{cases}
&\frac{1}{v(t,x)}\frac{\partial v}{\partial t}(t,x)=-\frac{1}{2}\big\{\frac{1}{v(t,x)}Tr[(\sigma\sigma^*(t,x)\nabla^2v(t,x))]\\
&\qquad\qquad\qquad\qquad -\frac{1}{v^2(t,x)}|\sigma^*(t,x) \nabla v(t,x)|^2+\frac{1}{v^v(t,x)}|\sigma^*(t,x)\nabla v(t,x)|^2\big\}\\
&\sigma(t,x)\sigma^*(t,x)\nabla v(t,x)=v(t,x)b(t,x),\quad (t,x)\in [0,\infty)\times \Lambda
\end{cases}
\endaligned\end{equation}
which are just (\ref{eq1.22}) and  (\ref{eq1.23}), respectively.
\end{proof}

\section{Non-Lipschitz SDEs with jumps}

\subsection{The characterisation theorem for SDEs with continuous diffusions on $\mR^d$}\label{chthrdc}
Let $(\mU,\|\cdot\|_{\mU})$ be
a finite dimensional normed space endowed with its Borel $\sigma$-algebra $\mathscr{U}$. Let $\nu$ be a $\sigma$-finite
measure defined on $(\mU,\mathscr{U})$. Let us fix $\mU_0\in\mathscr{U}$ with $\nu(\mU\setminus\mU_0)<\infty$
and $\int_{\mU_0}\|u\|_{\mU}^2\,\nu(\dif u)<\infty$. Furthermore, let $\lambda: [0,\infty)\times \mU\to(0,1]$ be a
given measurable function. Then, following e.g. \cite{iw,jjas}, there exists a non-negative  integer valued $(\cF_t)_{t\geq 0}$-Poisson
random measure $N_{\lambda}(\dif t, \dif u)$ on the given filtered probability space $(\Omega,\cF,\mP;(\cF_t)_{t\geq 0})$ with intensity
$\mE(N_{\lambda}(\dif t, \dif u))=\lambda(t,u)\dif t\nu(\dif u)$.
Set
$$\tilde{N}_\lambda(\dif t,\dif u):=N_\lambda(\dif t,\dif u)-\lambda(t,u)\dif t\nu(\dif u)$$
that is, $\tilde{N}_\lambda(\dif t,\dif u)$ stands for the compensated $(\cF_t)_{t\geq 0}$-predictable martingale
measure of $N_{\lambda}(\dif t, \dif u)$.

We are concerned with the following SDE on $\mR^d$
\be\left\{\begin{array}{l}
\dif X_t=b(t,X_t)\dif t+\sigma(t,X_t)\dif B_t+\int_{\mU_0}f(t,X_{t-},u)\tilde{N}_\lambda(\dif t, \dif u), \qquad t\in(0,T],\\
X_0=x_0\in \mR^d,
\end{array}
\right.
\label{Eq1}
\ee
for any given $T>0$, where $b, \sigma$ are Borel measurable as given in the previous section, $(B_t)_{t\geq0}$ is an $m$-dimensional $\{\cF_t\}_{t\geq0}$-Brownian motion,
$f: [0,T]\times\mR^d\times\mU_0\mapsto\mR^d$ is Borel measurable, and $\tilde{N}_\lambda$
is the compensated $(\cF_t)_{t\geq 0}$-predictable martingale measure of an induced $\{\cF_t\}_{t\geq0}$-Poisson random measure given above which is independent of
$(B_t)_{t\geq0}$. This equation arises in nonlinear filtering and has been considered
recently in \cite{hqxz,qiao,qd} (see also the monograph \cite{situ}).

The characterisation theorem for path-independent property of Girsanov density for the above equation with non-degenerated $\sigma$ was established in \cite{hqwu}. More precisely, under the following conditions
\begin{enumerate}[{\bf (H$_1$)}]
\item
There exists $\lambda_0\in\mR$ such that for all $x,y\in\mR^d$ and $t\in[0,T]$
\ce
2\<x-y, b(t,x)-b(t,y)\>+\|\sigma(t,x)-\sigma(t,y)\|^2\leq
\lambda_0|x-y|^2\kappa(|x-y|),
\de
where $\kappa$ is a positive continuous function, bounded on $[1,\infty)$ and satisfying
\ce
\lim\limits_{x\downarrow0}\frac{\kappa(x)}{\log
x^{-1}}=\delta<\infty.
\de
\end{enumerate}

\begin{enumerate}[{\bf (H$_2$)}]
\item
There exists $\lambda_1>0$
such that for all $x\in\mR^d$ and $t\in[0,T]$
$$
|b(t,x)|^2+\|\sigma(t,x)\|^2\leq \lambda_1(1+|x|)^2.
$$
\end{enumerate}

\begin{enumerate}[{\bf (H$_3$)}]
\item
$b(t,x)$ is continuous in $x$ and there exists $\lambda_2>0$ such that
\be
\<\sigma(t,x)h,h\>\geq\sqrt{\lambda_2}|h|^2, \qquad
t\in[0,T], \quad x, h\in\mR^d. \label{sig}
\ee
\end{enumerate}

\begin{enumerate}[({\bf H$_f$})]
\item
For all $x,y\in\mR^d$ and $t\in[0,T]$,
\ce \int_{\mU_0}\big|f(t,x,u)-f(t,y,u)\big|^2\nu(\dif
u)\leq2|\lambda_0||x-y|^2\kappa(|x-y|) \de and for $q=2$ and $4$
$$
\int_{\mU_{0}}|f(t,x,u)|^{q}\,\nu(\dif u)\leq \lambda_1(1+|x|)^{q}.
$$
\end{enumerate}
Qiao and Wu in \cite{hqwu} proved a characterisation theorem, where a partial integer-differential equation (PIDE) as the main characterizing equation was derived.
We notice that the assumption (H3) on the diffusion coefficient $\sigma$ is too
strong. Here we aim to relax this condition. First of all, we let
$\sigma$ to be $d\times m$-matrix-valued for $d,m\in\mN$, i.e., $\sigma$ is in
general not square matrix-valued.

Let $\gamma:[0,\infty)\times\mR^d\rightarrow\mR^m$ be a measurable function such that the following condition {\bf (H$_{\gamma,\lambda}$)} holds
\ce
&&\mE\Big[\exp\Big\{\frac{1}{2}\int_0^T\left|\gamma(s,X_s)\right|^2\dif s
+\int_0^T\int_{\mU_0}\left(\frac{1-\lambda(s,u)}{\lambda(s,u)}\right)^2\lambda(s,u)\nu(\dif u)\dif s\Big\}\Big]<\infty.
\de

Set
\ce
Z_t:&=&\exp\bigg\{-\int_0^t\<\gamma(s,X_s),\dif B_s\>-\frac{1}{2}\int_0^t
\left|\gamma(s,X_s)\right|^2\dif s\\
&&\quad\qquad -\int_0^t\int_{\mU_0}\log\lambda(s,u)N_{\lambda}(\dif s, \dif u)
-\int_0^t\int_{\mU_0}(1-\lambda(s,u))\nu(\dif u)\dif s\bigg\},\\
M_t:&=&-\int_0^t\<\gamma(s,X_s),\dif B_s\>
+\int_0^t\int_{\mU_0}\frac{1-\lambda(s,u)}{\lambda(s,u)}\tilde{N}_{\lambda}(\dif s, \dif u),
\de
and then $(Z_t)$ is the Dol\'eans-Dade exponential of $(M_t)$, see e.g., \cite{dm}.

Under ({\bf H}$_1$), ({\bf H}$_2$) and {\bf (H$_f$)}, it is well known that there exists
a unique strong solution to Eq.(\ref{Eq1}) (cf. \cite[Theorem 170, p.140]{situ}). This
solution will be denoted by $X_t$. In the following, we define the support of
a random vector (\cite{pa}) and then present a result about the support of $X_t$
under the above assumptions.

\bd\label{supset}
The support of a random vector $Y$ is defined as
$$
supp(Y):=\{x\in\mR^d | (\mP\circ Y^{-1})(B(x,r))>0, ~\mbox{for all}~ r>0\}
$$
where $B(x,r):=\{y\in\mR^d | |y-x|<r\}$, the open ball centered at $x$ with radius $r$.
\ed

Under {\bf (H$_{\gamma,\lambda}$)}, $(M_t)$
is a locally square integrable martingale. Moreover, $M_t-M_{t-}>-1$ a.s. and
\ce
&&\mE\Big[\exp\Big\{\frac{1}{2}<M^c,M^c>_T+<M^d,M^d>_T\Big\}\Big]\\
&=&\mE\Big[\exp\Big\{\frac{1}{2}\int_0^T\left|\gamma(s,X_s)\right|^2\dif s+\int_0^T\int_{\mU_0}\left(\frac{1-\lambda(s,u)}{\lambda(s,u)}\right)^2\lambda(s,u)\nu(\dif u)\dif s\Big\}\Big]<\infty,
\de
where $M^c$ and $M^d$ are continuous and purely discontinuous martingale parts of $(M_t)$, respectively.
Thus, it follows from \cite[Theorem 6]{ppks} that $(Z_t)$ is an exponential martingale. Define a measure $\tilde{\mP}$ via
$$
\frac{\dif \tilde{\mP}}{\dif \mP}=Z_t.
$$
By the Girsanov theorem for Brownian motions and random measures, one can obtain that under the measure
$\tilde{\mP}$ the system (\ref{Eq1}) is transformed into the following
\ce
\dif X_t=[b(t,X_t)+\sigma(t,X_t)\gamma(t,X_t)]\dif t+\sigma(t,X_t)\dif \tilde{B}_t+\int_{\mU_0}f(t,X_{t-},u)\tilde{N}(\dif t, \dif u),
\de
Now let us assume that (along th paths of $(X_t)_{t\geq0}$)
$$b(t,X_t)+\sigma(t,X_t)\gamma(t,X_t)=0.$$
Then we get
\ce
\dif X_t=\sigma(t,X_t)\dif \tilde{B}_t+\int_{\mU_0}f(t,X_{t-},u)\tilde{N}(\dif t, \dif u),
\de
where
\ce
\tilde{B}_t:=B_t+\int_0^t\gamma(s,X_s)\dif s, \quad
\tilde{N}(\dif t, \dif u):=N_\lambda(\dif t, \dif u)-\dif t\nu(\dif u).
\de

Next, we set
\ce
Y_t&:=&-\log Z_t=\int_0^t\<\gamma(s,X_s),\dif B_s\>+\frac{1}{2}\int_0^t
\left|\gamma(s,X_s)\right|^2\dif s\\
&&\quad +\int_0^t\int_{\mU_0}\log\lambda(s,u)N_{\lambda}(\dif s, \dif u)
+\int_0^t\int_{\mU_0}(1-\lambda(s,u))\nu(\dif u)\dif s.
\de
Clearly, $(Y_t)$ is a one-dimensional stochastic process with the following stochastic differential form
\ce
\dif Y_t&=&\<\gamma(t,X_t),\dif B_t\>+\frac{1}{2}\left|\gamma(t,X_t)\right|^2\dif t\\
&&\quad +\int_{\mU_0}\log\lambda(t,u)N_{\lambda}(\dif t, \dif u)
+\int_{\mU_0}(1-\lambda(t,u))\nu(\dif u)\dif t.
\de

Let $\Lambda:=\supp((t,X_t),t\geq0)$. Then we have the following.

\bt\label{chth1}
Let $v: [0,T]\times\mR^d\rightarrow\mR$ be a scalar function which is $C^1$
with respect to the first variable and $C^2$ with respect to the second variable. Then
\be
v(t, X_t)&=&v(0,x_0)+\int_0^t\<\gamma(s,X_s),\dif B_s\>+\frac{1}{2}\int_0^t
\left|\gamma(s,X_s)\right|^2\dif s\no\\
&&\quad +\int_0^t\int_{\mU_0}\log\lambda(s,u)N_{\lambda}(\dif s, \dif u)
+\int_0^t\int_{\mU_0}(1-\lambda(s,u))\nu(\dif u)\dif s,
\label{pathin}
\ee
equivalently,
\ce
Y_t=v(t,X_t)-v(0,x_0), \quad t\in[0,T]
\de
holds if and only if
\be
\gamma(t,x)&=&(\sigma\sigma^*\nabla v)(t, x), \qquad\qquad (t,x)\in \Lambda, \label{chco1}\\
\lambda(t,u)&=&\exp\{v(t,x+f(t,x,u))-v(t,x)\}, \quad (t,x,u)\in\Lambda\times\mU_0, \label{chco2}
\ee
and $v$ satisfies the following time-reversed partial integro-differential equation (PIDE),
\be
\frac{\partial}{\partial t}v(t,x)&=&-\frac{1}{2}[Tr(\sigma\sigma^*)\nabla^2 v](t,x)-\frac{1}{2}|\sigma^*\nabla v|^2(t,x)-\int_{\mU_0}\Big[e^{v(t,x+f(t,x,u))-v(t,x)}-1\no\\
&&\qquad\qquad -\<f(t,x,u),\nabla v(t,x)\>e^{v(t,x+f(t,x,u))-v(t,x)}\Big]\nu(\dif u).
\label{chco3}
\ee
\et

\begin{proof}
Following the line of \cite{qiao}. To the reader's convenience, we give the detailed proof here. The proof of necessity. By (\ref{pathin}),
\be
\dif v(t, X_t)&=&\left[\frac{1}{2}\left|\gamma(t,X_t)\right|^2+\int_{\mU_0}\Big(\lambda(t,u)\log\lambda(t,u)+\big(1-\lambda(t,u)\big)\Big)\nu(\dif u)\right]\dif t\no\\
&&\quad +\int_{\mU_0}\log\lambda(t,u)\tilde{N}_{\lambda}(\dif t, \dif u)+\<\gamma(t,X_t),\dif B_t\>.
\label{pathin1}
\ee
It is clear from (\ref{pathin1}) that $v(t, X_t)$ is a c\`adl\`ag semimartingale with
a predictable finite variation part. On the other hand, note that $X_t$ satisfies Equation (\ref{Eq1}) and $v(t, x)$ is a $C^{1,2}$-function, by applying the It\^o formula to the composition process $v(t, X_t)$, one could obtain the following
\be
\dif v(t, X_t)&=&\frac{\partial}{\partial t}v(t,X_t)\dif t+\<b,\nabla v\>(t, X_t)\dif t+\frac{1}{2}[Tr(\sigma\sigma^*)\nabla^2 v](t, X_t)\dif t\no\\
&&+\int_{\mU_0}\Big[v(t,X_{t-}+f(t,X_{t-},u))-v(t,X_{t-})\no\\
&&\qquad -\<f(t,X_{t-},u),\nabla v(t,X_{t-})\>\Big]\lambda(t,u)\nu(\dif u)\dif t\no\\
&&+\int_{\mU_0}\left[v(t,X_{t-}+f(t,X_{t-},u))-v(t,X_{t-})\right]\tilde{N}_\lambda(\dif t, \dif u)\no\\
&&+\<(\sigma^*\nabla v)(t,X_t),\dif B_t\>.
\label{itofor}
\ee
Thus, (\ref{itofor}) is another decomposition of the semimartingale $v(t, X_t)$. By uniqueness for decomposition of the semimartingale, it holds that for $t\in[0,T]$,
\ce
\gamma(t,X_t)&=&(\sigma^*\nabla v)(t,X_t),\\
\log\lambda(t,u)&=&v(t,X_{t-}+f(t,X_{t-},u))-v(t,X_{t-}), \quad u\in\mU_0,
\de
and
\ce
&&\frac{1}{2}\left|\gamma(t,X_t)\right|^2+\int_{\mU_0}\Big(\lambda(t,u)\log\lambda(t,u)+\big(1-\lambda(t,u)\big)\Big)\nu(\dif u)\\
&=&\frac{\partial}{\partial t}v(t,X_t)+\<b,\nabla v\>(t, X_t)+\frac{1}{2}[Tr(\sigma\sigma^*)\nabla^2 v](t, X_t)\no\\
&&+\int_{\mU_0}\Big[v(t,X_{t-}+f(t,X_{t-},u))-v(t,X_{t-})\no\\
&&\qquad -\<f(t,X_{t-},u),\nabla v(t,X_{t-})\>\Big]\lambda(t,u)\nu(\dif u), \quad a.s..
\de
Note that $(t,X_t)$ runs through $\Lambda$, thus, we have that
\be
\gamma(t,x)&=&(\sigma^*\nabla v)(t,x), \qquad\qquad\qquad\quad  (t,x)\in\Lambda, \label{chco11}\\
\log\lambda(t,u)&=&v(t,x+f(t,x,u))-v(t,x), \quad (t,x,u)\in\Lambda\times\mU_0, \label{chco22}
\ee
and
\be
&&\frac{1}{2}\left|\gamma(t,x)\right|^2+\int_{\mU_0}\Big(\lambda(t,u)\log\lambda(t,u)+\big(1-\lambda(t,u)\big)\Big)\nu(\dif u)\no\\
&=&\frac{\partial}{\partial t}v(t,x)+\<b,\nabla v\>(t,x)+\frac{1}{2}[Tr(\sigma\sigma^*)\nabla^2 v](t,x)\no\\
&&+\int_{\mU_0}\Big[v(t,x+f(t,x,u))-v(t,x)\no\\
&&\qquad -\<f(t,x,u),\nabla v(t,x)\>\Big]\lambda(t,u)\nu(\dif u).
\label{chco33}
\ee
It is easy to see that (\ref{chco11}) and (\ref{chco22}) correspond to (\ref{chco1}) and (\ref{chco2}), respectively,
which together with (\ref{chco33}) further yields the PIDE (\ref{chco3}).

Next, let us show sufficiency. Assume that there exists a $C^{1,2}$-function $v(t,x)$ satisfying
(\ref{chco1}), (\ref{chco2}) and (\ref{chco3}). For the composition process $v(t,X_t)$, the It\^o formula
admits us to get (\ref{itofor}). Combining (\ref{chco1}), (\ref{chco2}) and (\ref{chco3}) with (\ref{itofor}),
we have
\ce
\dif v(t, X_t)&=&\left[\frac{1}{2}\left|\gamma(t,X_t)\right|^2+\int_{\mU_0}\Big(\big(\lambda(t,u)\log\lambda(t,u)\big)\lambda(t,u)+\big(1-\lambda(t,u)\big)\Big)\nu(\dif u)\right]\dif t\no\\
&&\quad +\int_{\mU_0}\log\lambda(t,u)\tilde{N}_{\lambda}(\dif t, \dif u)+\<\sigma^{-1}(t,X_t)b(t,X_t),\dif B_t\>\\
&=&\<\gamma,\dif B_t\>+\frac{1}{2}
\left|\sigma^{-1}(t,X_t)b(t,X_t)\right|^2\dif t\no\\
&&\quad +\int_{\mU_0}\log\lambda(t,u)N_{\lambda}(\dif t, \dif u)
+\int_{\mU_0}(1-\lambda(t,u))\nu(\dif u)\dif t.
\de
The proof is completed.
\end{proof}

\medskip

\noindent {\bf Acknowledgments} The authors are grateful to Prof. Feng-Yu Wang and Dr Huijie Qiao for stimulating discussions.

\end{document}